\documentclass{amsart}
\usepackage{amssymb,amsfonts,amsmath,phonetic}
\usepackage{color}
\usepackage{graphicx}

\usepackage{ulem}
\def \emph{ \normalem }

\newtheorem{thm}{Theorem}[section]

\newtheorem{lem}[thm]{Lemma}
\newtheorem{prop}[thm]{Proposition}
\theoremstyle{definition}
\newtheorem{defn}[thm]{Definition}
\theoremstyle{remark}

\numberwithin{equation}{section}
\newcommand{\norm}[1]{\left\Vert#1\right\Vert}
\newcommand{\Norm}[1]{\lvert \! \lvert \! \lvert #1 \rvert \! \rvert \! \rvert}
\newcommand{\abs}[1]{\left\vert#1\right\vert}

\newcommand{\set}[1]{\left\{#1\right\}}
\newcommand{\parr}[1]{\left (#1\right )}
\newcommand{\brac}[1]{\left [#1\right ]}
\newcommand{\ip}[1]{\left \langle #1 \right \rangle }
\newcommand{\Real}{\mathbb R}

\newcommand{\too}{\rightarrow}

\newcommand{\A}{\mathcal{A}}

\newcommand{\bbar}[1]{\overline{#1}}
\newcommand{\wt}[1]{\widetilde{#1}} 
\newcommand{\wh}[1]{\widehat{#1}} 
\newcommand{\dpc}{\mbox{\rm{d}}_P}
\newcommand{\dcp}{\mathbf{D}_{cP}}
\def \vol{\mathrm{vol}}
\def \R{\mathcal{R}} 
\def \C{\mathcal{C}} 
\def \D{\mathcal{D}} 
\def \CC{\mathbb{C}} 
\def \Mob{\mbox{\it Mob}}
\def \A{\mathcal{A}}
\def \M{\mathcal{S}}
\def \N{{\mathcal{S}'}}
\def \NN{{\mathbb{N}}}
\def \L{\mathcal{S}''}
\def \pr{Pr} 
\def \bfi{\textbf{\footnotesize{i}}} 

\def \gd{\mbox{\bf{d}}} 

\newcommand{\grad}{\text{grad }}
\newcommand{\divergence}{\text{div }}

\definecolor{darkred}{rgb}{0.6,0,0}

\begin{document}
\title{The continuous Procrustes distance between two surfaces}
\author{Yaron Lipman, Reema Al-Aifari, Ingrid Daubechies \\ Princeton University }%
\begin{abstract}
The Procrustes distance is used to quantify the similarity or dissimilarity of (3-dimensional)
shapes, and extensively used in biological morphometrics. Typically each (normalized) shape
is represented by $N$ landmark points, chosen to be homologous (i.e. corresponding to each other), as much as possible, and the Procrustes
distance is then computed as $\inf_{R}\,\sum_{j=1}^N\, \|Rx_j-x'_j\|^2$, where the minimization is over all Euclidean transformations, and the correspondences $x_j \leftrightarrow x'_j$
are picked in an optimal way. \\
This (discrete)
Procrustes distance is easy to compute but 
has drawbacks -- representeding a shape by only a finite number of
points, which may fail to capture all the geometric aspects of interest; a need has been expressed
for alternatives that
are still computationally tractable.
We propose in this paper the concept of 
{\bf {continuous Procrustes distance}}, and prove that it provides a true metric
for two-dimensional surfaces embedded in three dimensions. The continuous Procrustes 
distance leads to a hard optimization
problem over the group of area-preserving diffeomorphisms. One of the core 
observations of our paper is that for small continuous Procrustes distances, 
the global optimum of the Procrustes distance can be uniformly approximated 
by a conformal map. This observation leads to an
efficient algorithm to calculate approximations to this new distance.
\end{abstract}
\maketitle


\section{Introduction}
Procrustes distances are used to compare shapes and quantify their
(dis)similarity. In several applications, such as geometric
morphometrics \cite{Mitteroecker09}, the shapes to be compared are
continuous surfaces, on each of which homologous landmark points are
selected, equal in number. The dissimilarity or distance between the
surfaces $\mathcal{S}$ and $\mathcal{S}'$ is then computed as the
{\it Procrustes distance} between their corresponding landmark
sequences $X=\left(x_{\ell}\right)_{\ell=1}^L$ and
$X'=\left(x'_{\ell}\right)_{\ell=1}^L$, which is defined as follows.
%
%
\begin{defn}
Given two finite sequences $X=\left(x_i\right)_{i=1}^n$,
$X' = \left(x'_i\right)_{i=1}^n$ in $\Real^d$, of equal length, with
centroids $\overline{x}$, $\overline{x'}$, and centroid sizes
$S_X$, $S_{X'}$, respectively\footnote{
The {\it centroid} of $X$ is given by $\overline{x}=n^{-1}\sum_{i=1}^n\,x_i\,$; the
{\it centroid size} by $S_X= \left[n^{-1}\sum_{i=1}^n\,(x_i-\overline{x})^2 \right]^{1/2}$
},
the classical Procrustes distance $d_P(X,X')$ between $X$ and $X'$ is defined by
\begin{equation}\label{e:discrete_Procrustes_classical}
    d_P(X,X') = \inf_{R\in \R}
\parr{\sum_{i=1}^n \norm{
    \frac{Rx_i}{S_X}-\frac{x'_i}{S_{X'}}}^2}^{1/2},
\end{equation}
where $\R$ is the group of Euclidean
transformations
(reflections,rotations, and translations).
\end{defn}

In some applications, it may be useful to consider {\it weighted
Procrustes distances}, in which each label $i \in \{1,\ldots,L\}$
can be given its own weight $w_i$ in the computation of the
centroids, the centroid sizes and the distance $d_P(X,X')$; such
weighting can be used to compensate, if desired, for possible
imbalances in the distribution of the landmark points, when they
occur more densely in some areas than in others. We shall assume in
what follows that no such adjustment is needed, i.e., that the
landmark points are considered (more or less) uniformly distributed.
The normalization by the centroid size allows comparison of shapes
irrespective of their scale. To achieve this without a normalization
step, one would need to extend $\R$ to the larger group of
similarities, incorporating the (uniform) dilations as well.
Note that other geometric extensive quantities
could be used to normalize, with a very similar effect.

The point sets $X$ and $X'$ are said to have the same shape if
one can be obtained from the other by an appropriate combination of scaling, translating,
rotating and (possibly) mirroring, i.e., if there exists $R\in \R$ and $\alpha \in \Real$
such that $\alpha RX=X'$. (``Shapes'' of finite sets of points
can thus be considered as orbits
of these point sets under the action of similarity operations.)
It is not hard to check that the Procrustes distance $d_P(\cdot,\cdot)$
is a metric on shapes, in the sense that it satisfies, for all finite point sequences $X$,
$X'$ and $X''$
in $\Real^d$ of equal length,\\
1) $d_P(X,X')=d_P(X',X)$, \\
2) $d_P(X,X')\geq 0 $, and $d_P(X,X')=0$ implies
that $X$ and $X'$ have the same shape, \\
3) $d_P(X,X'')\leq d_P(X,X') + d_P(X',X'')$.

In all the above, the points $x_i$ and $x'_i$ are ordered, i.e. corresponding points in
$X$ and $X'$ have the same index. The correspondence between entries of $X$ and $X'$ can
be ``encoded'' by a bijective
correspondence map $\C: X \too X'$ that maps each
$x_j \in X$ to its corresponding $x'_j \in X'$. In terms of this correspondence map the
Procrustes distance $d_P(X,X')$ can be written as
\begin{equation}\label{e:discrete_Procrustes_with_corrmap}
    d_P(X,X') = d_P(X,\C X) = \inf_{R\in \R}\, \Norm{RX-\C X}
   := \inf_{R\in \R} \parr{\sum_{i=1}^n \norm{R
    x_i-\C x_i}^2}^{1/2};
\end{equation}
this recasts the minimization as a search for the map in $\R$ that best
approximates $\C$, insofar as its action on $X$ is concerned.

Using a finite set $X$ of landmark points $x_{\ell} \in \mathcal{S},\,\ell=1,\ldots,L$ as a proxy
for the shape of a surface $\mathcal{S}$, and taking the value of $d_P(X,X')$ to express the
dissimilarity of the shapes of surfaces $\mathcal{S}$ and $\mathcal{S}'$ has some drawbacks, however.
First, this approach compares only small discrete subsets of points sampled
from the surfaces and therefore ignores ``most'' of their shapes.
Second, and more importantly, it requires the user to carefully
select corresponding landmark points on the two surfaces prior to
calculating the Procrustes distance between the two landmark point sequences.
This distance depends heavily on the exact choice of the landmark points. In
geometric morphometrics, one seeks to remove some of the arbitrariness of
these choices by picking landmark points that are believed to be {\it homologous}, i.e.,
to truly correspond to each other, based on evolutionary arguments.
This type of selection of landmark points requires considerable specialized expertise,
and in some cases, even experts do not agree. In addition, morphologists interested
in studying function of e.g. teeth are interested in moving away from landmark selection,
and in using geometric information that encompasses more global features.

This situation has motivated researchers to suggest alternative methods
to compute distances or dissimilarities between the shapes of surfaces.
Even when these methods are based on continuous concepts, their numerical
implementation requires some type of discretization, and thus often involves
again discrete point sets $X$ and $X'$ (typically of larger cardinality than
in landmark-based distances).
The resulting
distance can then still be written in the same form as the right hand side of
(\ref{e:discrete_Procrustes_with_corrmap}), with the important difference that
$\C$ is no longer assumed to be given {\it a priori}. Instead, the map $\C$ is
assumed to be determined by the full geometry of the surfaces $\mathcal{S}$ and $\mathcal{S}'$;
in practice, it has to be derived from the data themselves, meaning that both the
correspondence $\C$ and the Euclidean transformation $R$ must be determined
numerically. (One could imagine a similar situation in the discrete case, if two sets
$X$ and $X'$ were given, each with $n$ points, without a correspondence map. In that case,
a reasonable approach might be to select the map $\C: X \too X' \,$ for which
$\Norm{\C X - X'}$ is smallest.)

A prominent method of this type is
the Iterative Closest Point (ICP) algorithm
\cite{Besl92}. This method alternates between determining $\C$ and $R$:
the correspondence
$\C_k:X\too X'$ is taken to associate to each point
$x' \in X'$ the point(s) in $X$ for which the image $R_{k-1} x$, under
the best rigid alignment of $X$ and $X'$ obtained in the previous iteration,
is closer to $x'$ than to any other element in $X'$; the rigid alignment $R_k$
is then the transformation $R \in \R$ that minimizes the distance $\Norm{RX - \C_k X}$.
This algorithm is simple and robust, but
suffers from several drawbacks. It may converge to a local
minimum rather than the desired minimizer; this means that the limit
may depend on the choice of the
initial correspondence map $\C_1$ or the initial rigid alignment $R_0$, whichever is
picked to start off the algorithm.
Of more concern is that the space of
possible correspondences $\C:\mathcal{S}\too\mathcal{S'}$ considered by the algorithm
consists only of compositions of rigid
motions and closest neighbor maps. This space of maps often
contains high-distortion and discontinuous mappings, as illustrated
in Figure \ref{fig:diffeo_vs_icp} for the 1-dimensional situation; it also
does not include a sufficiently rich set of diffeomorphisms
(smooth bijective mappings).

\begin{figure}[h]
\begin{centering}
\includegraphics[width=0.75 \columnwidth]{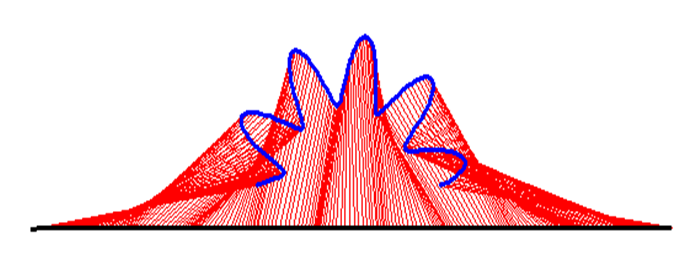} \\
(a) The correspondence map $\C_1:X \too X'$, in this case a
length-preserving diffeomorphism. ($\M$ and $\N$
each consist of 200 points, equispaced on the black horizontal line
($X$) and on the blue curve ($X'$),
respectively.)\\
\includegraphics[width=0.75 \columnwidth]{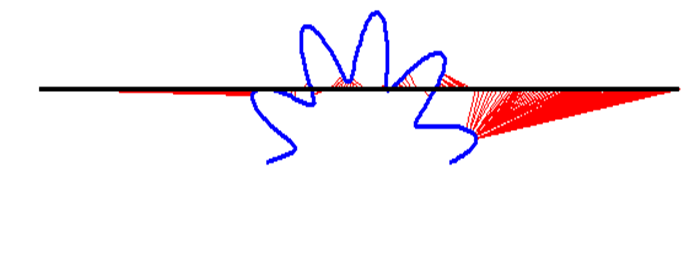}
\vspace*{-1.2 cm}

(b) Illustration of the Euclidean map $R_1$ (moving the black line ``up'')
and the correspondence map $\C_2$.\\
\end{centering}
\caption{{\bf One-dimensional illustration of one step in the ICP algorithm.}
$X$ and $X'$ are point sets, each with 200 points,
on two curves (``one-dimensional surfaces'') $\mathcal{S}$ (straight, black)
and $\mathcal{S}'$ (wiggly, blue) of equal length. 
(a) The correspondence (in red) between $X$ and $X'$ that associates to
each point $x' \in X'$ the point $x \in X$ at the same arclength
distance from the left end point of its curve. 
(b) Using this correspondence as an initial $\C_1:X \too X'$,
determine the Euclidean transformation (now a simple translation in the plane) $R_1$
that minimizes $\Norm{RX-\C_1X}$, and move $X$ to $R_1X$; the red lines now
link each $R_1x \in R_1X$ to the closest point in $X'$.
The corresponding map $\C_2: X \too X'$
is discontinuous and highly distorting.}
\label{fig:diffeo_vs_icp}
\end{figure}

Several authors have built extensions or generalizations of
this approach, retaining the basic iterative principle of ICP,
interleaving the determination of
correspondences $\C_k$ and
transformations $R_k$ in successive steps. \\
Rangarajan et al.~\cite{Rangarajan97thesoftassign} formulate a
variant on the Procrustes distance between two discrete sets of
points in which the correspondence maps are unknown {\it a priori}.
Their algorithm alternates between calculating optimal rotations
and determining correspondence maps (bi-measures). For every fixed
rotation $R$, it computes 
the ``measure coupling numbers'' $M_{ij}$ from one point set to another,
minimizing the average of the squared residuals $\sum_{i,j} M_{ij}
\norm{Rx_i - x'_j}$, under the (soft) constraint that 
$\left(M_{ij}\right)_{i,j=1,\ldots,n}$ is
indeed a measure coupling. 
As is the case with ICP, this algorithm can still converge to a local 
rather than a global minimum, and
the correspondence maps can
still be ``discontinuous and/or distorting''. Ghosh et
al.~\cite{ghosh:2009} use a similar framework (although not related
to Procrustes or any other distance) with a smooth
surface deformation mechanism together with closest point maps to
determine both the correspondence maps and the transformations in an alternating
iterative procedure.
The algorithm in \cite{ghosh:2009} requires user initialization (which may influence the outcome);
the way
correspondences are assigned can lead the deformation mechanism to ultimately produce
a distorting and/or discontinuous map between the surfaces.

A common characteristic of the algorithms mentioned above, which
often (in the limit or in intermediate stages) lead to discontinuous
or distorting correspondence maps, is that the space they explore
(implicitly or explicitly) to build correspondence maps is
insufficiently rich in smooth bijections.

In this paper we generalize the discrete Procrustes distance
to continuous surfaces; in this formulation we use only {\it smooth} correspondence maps. 
Our construction
leads to a non-linear functional over a huge and non-linear space of
possible maps that we call the {\it continuous Procrustes functional}. 
Direct optimization, over its huge domain, of this generalized
Procrustes functional is not feasible; we suggest that for many cases
of interest, a different optimization suffices, over
a (relatively) much smaller (and managable) subset of all possible maps,  consisting
of conformal mappings combined with specific area-preserving maps. One of the main results
of our paper is a proof that the class of conformal maps
{\it uniformly} approximates the globally optimal correspondence
map in the regime where the continuous Procrustes functional takes on
small values. Note that our approach thus provides a glimpse of
the global minimizer (for the case of small continuous Procrustes
distance) of a functional for which it is not known, in general, 
how to approximate the
global minimizer in polynomial time. \\
In addition to the theoretical
constructions, we also provide an efficient algorithm, without 
user interaction, to construct (an approximation to) the continuous Procrustes 
distance and the optimal correspondence map between two surfaces, again in the case where this
distance is small, i.e. where the surfaces are not too dissimilar. 
In practice, this algorithm performs very well, and is sufficiently fast to 
be used for the computation of pairwise distances for all pairs in reasonably large
collections of surfaces ($\sim 100$); see \cite{lipman_at_nordia_2010}, a first presentation of
the main results of this paper at a workshop in June 2010, as well as \cite{pnas}, which
uses the algorithm explained here in detail for three biological data-sets. 

A similar combination (conformal mappings composed with
area-preserving maps) is used in the recent paper by Dominitz and
Tannenbaum \cite{Dominitz10}, to construct good mappings from
surfaces to a Euclidean spherical domain. The goal of
\cite{Dominitz10} is different, however; rather than seeking to
define a distance between surfaces, that can be used for shape
alignment, \cite{Dominitz10} is concerned with building a low
distortion map from a surface to Euclidean domain, the inverse of
which can then be used as a good parameterization for the surface.

The paper is organized as follows. In section 2, we introduce our
definition of the continuous Procrustes distance for homeomorphic
2-dimensional compact surfaces embedded in $\Real^3$; it involves a
minimization that is unfeasible in practice. In section 3, we show
that we can construct approximations of this distance by minimizing over appropriate
perturbations of conformal mappings, which is much more tractable.
In section 4, we give the corresponding numerical algorithm and 
illustrate them with a concrete example.

\section{The Surfaces Procrustes Distance.}
Consider two homeomorphic
compact 2-dimensional surfaces $\M, \, \N$ embedded in $\Real^3$,
%
%
endowed with the standard metric induced from their
embedding. Because the biological applications that motivated this work
require comparing shapes irrespective of scale (see \cite{Mitteroecker09} and reference
therein),
we are interested in defining a scale-independent distance. We shall
therefore assume that the two surfaces are normalized to have
unit volume
(area):
\[
\int_\M d\vol_\M = 1 = \int_{\N} d\vol_{\N}.
\]
In building ``good'' correspondence maps between the continuous surfaces,
we will be guided by what represents a``good'' correspondence between discrete
(relatively small) sets of points
that ``represent'' the surfaces when standard Procrustes distances are used. \\
As mentioned in the introduction, great care is typically taken
in the choice of sample points on surfaces that will then be used
to compare the
shapes of these surfaces. Landmark points on e.g. teeth or other bones
are chosen so that they
are homologous, i.e. ``equivalent'' from an evolutionary point of view.
We are aiming for a landmark-free method;  information of this type will thus
not be available. Instead, we can use only geometric information given by the
surface itself. Note that in (\ref{e:discrete_Procrustes_classical}), the different
points $x_i$ all play an equal role.
When choosing discrete sets $X$, $X'$, each consisting of $n$ points, on the surfaces
$\M$ and $\N$ to represent their respective ``shapes'', with the purpose of using them in a
Procrustes distance calculation (\ref{e:discrete_Procrustes_classical}),
it seems therefore reasonable to pick the points so that
each represents an equal ``share'' of the surfaces; we shall interpret
this here as representing an equal portion of the {\it area} of the surfaces.
A correspondence map $\C$ that maps each $x_i$ to its partner $x_i'$ can thus be interpreted
as mapping portions of area $1/n$ of $\M$ to the corresponding
portions of $\N$ that have equal area; on the other hand, the sum
in (\ref{e:discrete_Procrustes_classical}) can be viewed (up to a normalization) as a
Riemann sum approximation to the integral of $\norm{Rx-\C x}$ over $\M$.

This analysis suggests the following ``continuous analogue'' of the discrete
construction.
To involve the whole surface $\M$
(instead of just a set of sample points), we take $\C$ to be an area-preserving map
from $\M$ to $\N$;  for each fixed area-preserving $\C$, we then define
\begin{equation}
\dpc (\M,\N;\C)=\inf_{R \in \R}\,\left(\, \int_{\M} \, \norm{Rx-\C x }^2 \,d\vol_{\M}(x)\, \right)^{1/2}\,.
\label{eq:dpc}
\end{equation}
In the absence of landmark-type or other user-guided information we
have to select $\C$ based solely on geometric information. Taking
our cue from the discrete case, we want, given points
$\left(x_i\right)_{i=1,\ldots,n}$ on $\M$, to determine
$\left(x'_i\right)_{i=1,\ldots,n}$ on $\N$ so that each $x'_i$
corresponds ``as well as possible'' to $x_i$. In other words, this
suggests that $\C$ be picked so that $\N$ and $\C \M$ are optimally
aligned, and that the continuous Procrustes distance be given by the
corresponding value of $\dpc(\M,\N;\C)$. More explicitly, defining
$\A(\M,\N)$ to be the set of diffeomorphisms (smooth bijective maps
with a smooth inverse) from $\M$ to $\N$ that are area-preserving,
we set
\begin{equation}
\dcp(\M,\N)= \inf_{\C\in \A(\M,\N)} \,\left[\,
\inf_{R \in \R}\,\left(\, \int_{\M} \, \norm{Rx-\C x }^2 \,d\vol_{\M}(x) \, \right)^{1/2}
\, \right]\,.
\label{eq:cont_P_dist}
\end{equation}

In the remainder of this section, we establish several properties for the quantities defined in
(\ref{eq:dpc}) and (\ref{eq:cont_P_dist}), establishing, among other results, that
$\dcp(\cdot,\cdot)$ defines a metric distance.

We start by proving that the minimum in (\ref{eq:dpc}) is always
attained.
\begin{prop}\label{prop:minimum_R_attained}
Given two homeomorphic surfaces $\M$, $\N$ of unit area, and an area-preserving map $\C$ from
$\M$ to $\N$, there exists a rigid motion $R^*\in \R$ minimizing the functional
$\int_{\M} \, \norm{Rx-\C x }^2 \,d\vol_{\M}(x)\,$.
\end{prop}
\begin{proof}
Let $R_n \in \R$ be a sequence such that
$$\parr{\int_{\M}\norm{R_nx - \C x}^2 d\vol_\M(x)} - d_{P}(\M,\N; \C)^2 < \frac{1}{n}.$$
 Let us represent each rigid motion as a composition
of an orthogonal transformation and a translation: $$R_n x = U_n x +
t_n,$$ where $U_n \in \Real^{3\times 3}$ and $t_n\in \Real^3$.
Thinking of $R_n=(U_n,t_n)$ as a vector in $\Real^{12}$ it is clear
that there exists some compact set $A\subset \Real^{12}$ such
that $R_n \in A$ for all $n$. Indeed, the orthogonal group $O(3)$ in
its representation as a $3\times 3$ matrix group is a compact set,
and for sufficiently large  $n$, the $t_n$ will all lie within some ball,
i.e. $\norm{t_n}<M$ for some $M$. Hence there exists some rigid transformation
$R^*=(U^*,t^*)$ such that, up to extracting a subsequence, $R_n\too R^*$
as $n \too \infty$. Lastly, $R^*$ realizes the infimum $d_{P}(\M,\N;
\C)$ since $\norm{U_n x + t_n - \C x} \too \norm{U^* x + t^*
- \C x}$ for every $x \in \M$, and similar arguments as above imply that
$\norm{U_n x + t_n - \C x}$ is bounded uniformly in $n$ and $x\in \M$.
The result then follows from the dominated convergence theorem.
\end{proof}

The following two propositions provide closed form solutions for
(\ref{eq:dpc}); their proofs follow the discrete case
\cite{Eggert97} in a rather straightforward manner.
Note that we use that $\C$ is area-preserving
to establish these formulas (for the translation part).
First, we show that the translational part $t^*$
takes the centroid of $\M$ to the centroid of $\N$:
\begin{prop}\label{prop:translational_part}
If $\M$ and $\N$ both have their centroids at the origin, i.e. $\int_\M x\, d\vol_\M(x) = 0 =
\int_{\N} y \,d\vol_{\N}(y)$, then the translational part $t^*$ of the optimal
rigid motion $R^*$ is zero.
\end{prop}
\begin{proof}
Assume the surfaces $\M, \, \N$ are centered as described in the
assumptions of the theorem. Differentiating $\int_\M \norm{Ux+t -
\C x}^2 d\vol_\M(x)$ with respect to each of the coordinates of the
vector $t$ and plugging in $U=U^*$ and $t=t^*$, we get
$$0 = 2\int_\M \parr{U^*x+t^*-\C x} d\vol_\M(x).$$
Rearranging the above equality and remembering that $\M$ and $\N$
have unit area, we get
\begin{align*}
t^* & =  \int_\M  \C x \ d\vol_\M(x) -  \int_\M  U^*x \ d\vol_\M(x) \\
& =  \int_{\N}  y \ d\vol_{\N}(y) -  U^* \int_\M  x \ d\vol_\M(x) = 0 .
\end{align*}
\end{proof}
Next, the orthogonal transformation part:
\begin{prop}\label{prop:rotational_part}
If $\M$ and $\N$ both have their centroids at the origin, i.e. $\int_\M x d\vol_\M(x) = 0 =
\int_{\N} y d\vol_{\N}(y)$, then the optimal orthogonal transformation $U^*$
can be written as
\begin{equation*}\label{e:optimal_orth_trans}
    U^* = WQ^T,
\end{equation*}
where $W,Q$ are the orthogonal transformations from the Singular Value Decomposition (SVD)
\begin{equation*}\label{e:continuous_correlation_matrix_SVD}
    \int_\M x \parr{\C x}^T d\vol_\M (x) = Q S W^T,
\end{equation*}
where $S=diag\parr{\sigma_1,\sigma_2,\sigma_3}$ is a diagonal
matrix with the singular values of $U^*$ on the diagonal.
\end{prop}
\begin{proof}
Expanding $\int_\M \norm{U x-\C x}^2 d\vol_\M(x)$ we get:
\begin{align*}
&\int_\M \norm{U x-\C x}^2 d\vol_\M(x) =   \\  & \int_\M \norm{x}^2
d\vol_\M(x) - 2\int_\M x^T U^T \C x \ d\vol_\M(x) + \int\norm{\C x}^2
d\vol_\M(x),
\end{align*}
where we used that $U^T U = Id$. The sought-for $U^*$
therefore must maximize
\[
E(U) = \int_\M x^T U^T \C x \ d\vol_\M(x).
\]
Note that
$$x^T U^T \C x = tr \parr{x^T U^T \C x} = tr \parr{U x \parr{\C x}^T},$$
and therefore
\begin{align*}
E(U) & = tr \brac{ U \int_\M x \parr{\C x}^T  \ d\vol_\M(x)}
= tr \brac{U Q S W^T} \\ & = tr \brac{W^T U Q S} = tr \brac{\wt{U} S}
 = \wt{U}_{1,1} \sigma_1 + \wt{U}_{2,2} \sigma_2 + \wt{U}_{3,3}
\sigma_3,
\end{align*}
where $\wt{U} := W^T U Q$, and we used the SVD decomposition.
 Note that the last term cannot be greater than
$\sigma_1+\sigma_2+\sigma_3$ since all entries of the orthogonal
matrix $\wt{U}$ have absolute value at most 1.
Further note that taking $U = W Q^T$ achieves this upper bound. The
uniqueness is also clear.
\end{proof}

We now prove:
\begin{prop}
For each fixed area-preserving map $\C$ from $\M$ to $\N$, we have

\hspace*{.5 cm}{\rm (1)} $\dpc(\M,\N;\C)\geq 0 $\\
\hspace*{.5 cm}{\rm (2)} $\dpc(\M,\N;\C)\,= \,\dpc(\N,\M;\C^{-1})$ \\
\hspace*{.5 cm}{\rm (3)} $\dpc(\M,\N;\C)\,= 0$  implies that $\M$ and $\N$ are congruent.

Moreover, if $\L$ is a third surface, and $\C'$ is an area-preserving map from $\N$ to $\L$, then

\hspace*{.5 cm}{\rm (4)} $\dpc(\M,\L;\C'\circ \C)\leq \dpc(\M,\N;\C)\,+\, \dpc(\N,\L;\C')$.
\label{thm:continuous_proscrustes_is_a_metric}
\end{prop}
\begin{proof}
First, it is clear that
$\dpc(\M,\N ; \C) \geq 0$.
If $\dpc(\M,\N ; \C) = 0$, then we know by Proposition \ref{prop:minimum_R_attained}
that there exists a rigid transformation $R^*$ such that
$$\int_\M \norm{R^* x - \C x}^2d\vol_\M(x) = 0.$$
Since we
are dealing with smooth surfaces $\M$, this implies that $\norm{R^*x-\C x}=0$ for all $x\in \M$;  since the range of $\C$ is all of $\N$ (because $\C$ is a bijective
diffeomorphism) it follows that $R^*\M$ and $\N$ are equal as sets, so
$\M$ and $\N$ are congruent.

Next we prove symmetry:
\begin{align*}
\dpc(\M,\N;\C)^2 & = \inf_{R\in\R} \int_{\M}\norm{Rx - \C x}^2 d\vol_{\M}(x)
= \inf_{R\in\R} \int_{\N}\norm{R\C^{-1}y - y}^2 d\vol_{\N}(y) \\ & =
\inf_{R\in\R} \int_{\N}\norm{\C^{-1}y - Ry}^2 d\vol_{\N}(y) = \dpc(\N,\M;
\C^{-1})^2,
\end{align*}
where the second equality uses the fact that $\C$ is
area-preserving.

For arbitrary $\wt{R}\in\R$, we have
\begin{align*}
 & \dpc(\M,\L;\C'\circ \C) = \inf_{R\in\R} \parr{\int_{\M}\norm{Rx -
\C' \circ \C x}^2 d\vol_\M(x)}^{1/2} \\ & \leq \inf_{R\in\R}
\set{\parr{\int_{\M}\norm{Rx - \wt{R}\C x}^2 d\vol_\M(x)}^{1/2} }+
\parr{\int_{\M}\norm{\wt{R}\C x - \C'\circ \  \C x}^2 d\vol_\M(x)}^{1/2}
\\ & =
\inf_{R\in\R} \set{\parr{\int_{\M}\norm{Rx - \C x}^2
d\vol_\M(x)}^{1/2}} +
\parr{\int_{\N}\norm{\wt{R}y - \C' y}^2 d\vol_{\N}(y)}^{1/2}
\\ & =
\dpc(\M,\N; \C) +
\parr{\int_{\N}\norm{\wt{R}y - \C' y}^2 d\vol_{\N}(y)}^{1/2}.
\end{align*}
By taking the infimum over all $\wt{R} \in \R$ we obtain the desired
result.
\end{proof}

Having established these properties for $\dpc(\M,\N;\C)$, we can now
minimize this over $\C \in \A(\M,\N)$, i.e. we have
\begin{equation}\label{e:procrustes_and_correspondence}
    \dcp (\M,\N) = \inf_{\C\in\A(\M,\N)} \dpc (\M,\N ; \C);
\end{equation}
if the infimum is achieved by some $\C^*\in\C$, we declare this $\C^*$
to be our desired correspondence map.

Whether such a minimizer exists is a delicate question, and we do
not have a proof or counter example for the general case. However, if
we restrict the class of maps $\C$ to
bi-Lipschitz maps with some {\it a priori} bound on the  maximal dilation,
then such a
minimizer does indeed exist; moreover this minimizer is also bi-Lipschitz with the same
bound.

\begin{prop}
For arbitrary $B>0$,  
let $\mathcal{B}_B(\M,\N)$ be the set of bi-Lipschitz area-preserving
diffeomorphisms from $\M$ to $\N$ such that, for all $x$, $y \in
\M$, $B^{-1}\, \gd_{\M}(x,y) \leq \gd_{\N}(\C x, \C y) \leq B
\gd_{\M}(x,y)$,  and let $\gd_{\M}$, resp.
$\gd_{\N}$ denote the geodesic distances on $\M$, resp. $\N$. Then
there exists a minimizer in $\mathcal{B}_B(\M,\N)$ for the
restriction to $\mathcal{B}_B(\M,\N)$ of the functional $\dcp (\M,\N
; \cdot )$.
\end{prop}
\begin{proof}
It is straightforward that $\mathcal{B}_B(\M,\N)$ is a closed subset
of $C(\M,\N)$, the set of continuous functions from $\M$ to $\N$,
equipped with the topology of uniform convergence with respect to
$\gd_{\M}$ and $\gd_{\N}$. By the definition of
$\mathcal{B}_B(\M,\N)$, the functions in
$\mathcal{B}_B$ are equicontinuous. It then follows from the Ascoli-Arzela theorem for the continuous functions on compact metric spaces that $\mathcal{B}_B(\M,\N)$ is compact. \\
It is also easy to see that the functional $\dpc (\M,\N ; \cdot )$ is continuous with respect to the
topology of uniform convergence on $C(\M,\N)$. It follows that the restriction of
$\dpc (\M,\N ; \cdot )$ to $\mathcal{B}_B(\M,\N)$ is a continuous map from a compact space to
$\Real$. Let now $\left(\C_n\right)_{n \in \mathbb{N}}$ be a minimizing sequence in
$\mathcal{B}_B(\M,\N)$, i.e. $\dpc (\M,\N ; \C_n ) \too \inf_{\C \in \mathcal{B}_B(\M,\N)}
\dpc (\M,\N ;\C)$ as $n\too \infty$. By the compactness of $\mathcal{B}_B(\M,\N)$, the sequence
$\left(\C_n\right)_{n \in \mathbb{N}}$ has a uniformly converging subsequence; if we denote its limit
by $\C^*$, then it follows that $\C^* \in \mathcal{B}_B(\M,\N)$, and
$\dpc (\M,\N ;\C^*)=  \inf_{\C \in \mathcal{B}_B(\M,\N)} \dpc (\M,\N ;\C)$.
\end{proof}

We note here that all the further
proofs and results in the paper will remain valid ({\it mutatis mutandum})
if we replace everywhere the class of general area-preserving maps by
the more restricted class of bi-Lipschitz area-preserving maps.

Even when the existence of a minimizer is not guaranteed,
it is possible to prove that $\dcp(\M,\N)$ defines a
metric up-to congruence relation:
\begin{thm}
$\dcp(\M,\N)$ defines a metric between surfaces up-to-congruence, that is,
$\dcp(\M,\N) \geq 0$, $\dcp(\M,\N)=\dcp(\N,\M) $, $\dcp(\M,\N) \leq \dcp(\M,\L) + \dcp(\L,\N)$, and
$\dcp(\M,\N) = 0$ only if $\M$ and $\N$ are congruent.
\end{thm}
\begin{proof}
Clearly $\dcp(\M,\N) \geq 0$.

If $\dcp(\M,\N)=0$ then we have a sequence $(\C_n)_{n\in \mathbb{N}}
\subset \A(\M,\N)$ and (by Proposition
\ref{prop:minimum_R_attained}) a sequence $\Big(R_n\Big)_{n\in
\mathbb{N}} \subset \R$ of rigid motions such that:
$$\int_\M \norm{R_n x - \C_n x}^2 d\vol_\M(x) < \frac{1}{n^3}.$$
By extracting a subsequence $\Big(R^*_k\Big)_{k\in \mathbb{N}}$ (with $R^*_k:=R_{n_k}$), 
we can assume that $\norm{R^*_k - R^{**}}_{\infty}\too 0$, with $R^{**} \in \R$ 
as $k \too \infty$, and that (with $\C^*_k:=\C_{n_k}$) 
$$ 
\int_\M \norm{R^{**}x - \C^*_k x}^2 d\vol_\M(x) < \frac{1}{k^3}.
$$
Set now $B_{n,\ell}:=\{x \, ;\, \norm{R^{**}x - \C^*_n x}^2\geq
1/\ell\}$, $ A_{\ell}:=\{x \,;\, \, \limsup_{k\too
\infty}{\norm{R^{**}x - \C^*_k x}^2}\geq 1/\ell\} = \cap_{m \in \NN}
\cup_{n \geq m}B_{n,\ell}\,, $ and $ A_{\infty}:=\cup_{\ell \geq
1}A_{\ell} $ $= \{x\, ;\, \, \limsup_{k\too \infty}{\norm{R^{**}x -
\C^*_k x}}\neq0\}$. Since $ A_{\ell}\subset A_{\ell+1} $ for all
$\ell$, it follows that $\vol_{\M}(A_{\infty})$ $= \lim_{\ell \too
\infty}\vol_{\M} (A_{\ell})\,$. We have $\vol_{\M}(B_{n,\ell})$$\leq
\ell \, \int_\M \norm{R^{**}x - \C^*_n x}^2 d\vol_\M(x)$ $ \leq
{\ell}/{n^3}\,,$ hence $\vol_{\M}\Big(\cup_{n \geq
m}B_{n,\ell}\Big)\leq {\ell}/{m^2}\,$ and thus $\vol_{\M}(A_{\ell})$
$= \vol_{\M}\Big(\cap_{m \in \NN} \cup_{n \geq m}B_{n,\ell}\Big)
\leq \inf_{m\in \NN} {\ell}/{m^2} =0\,$. It follows that
$\vol_{\M}(A_{\infty})=0$, or $ \vol_{\M}\Big(\{x\, ;\, \,
\limsup_{k\too \infty}{\norm{R^{**}x - \C^*_k x}}\neq
0\}\Big)\,=\,0. $

Therefore $R^{**}x=\lim_{k \too \infty} \C_k(x)$ for $x \in \M
\setminus A_{\infty}$, implying $ R^{**}(\M \setminus
A_{\infty})\subset \overline{\N}=\N $. Since every open disk in $\M$
(with respect to the geodesic distance on $\M$) has area strictly
greater than 0 in $\M$, $\M \setminus A_{\infty}$ is dense in $\M$.
By the continuity of $R^{**}$ it follows that $R^{**}(\M) \subset
\N$.

Let's assume now (hoping to derive a contradiction) that
there exists a point $\wt{y}\in \N$ such that $\wt{y}
\notin R^{**}(\M)$. Since $R^{**}(\M)$ is a closed set there must then exist
a set $O \subset \N$ with positive area such that $O \cap R^{**}(\M)
= \emptyset$. This is a contradiction since $R^{**}:\M\too\N$ is an
isometry and in particular area-preserving. Hence $R^{**} (\M) =
\N$, showing that $\M$ and $\N$ are congruent.

Symmetry is easy to establish as follows:
\begin{align*}
\dcp(\M,\N) & = \inf_{\C\in \A(\M,\N)}\dpc(\M,\N; \C) = \inf_{\C\in
\A(\M,\N)}\dpc(\N,\M; \C^{-1}) \\ & = \inf_{\C\in \A(\N,\M)}\dpc(\N,\M;
\C) = \dcp(\N,\M),
\end{align*}
where we used that $\C\in\A(\M,\N)$ iff $C^{-1}\in\A(\N,\M)$.

Lastly, for the triangle inequality, we have, by Theorem
\ref{thm:continuous_proscrustes_is_a_metric}, for every $\C\in
\A(\M,\N)$ and every $\C' \in \A(\N,\L)$,
$$
\dcp(\M,\L)  = \inf_{\C''\in \A(\M,\L)}\dpc(\M,\L; \C'') \leq \dpc(\M,\L; \C'\circ \C) \leq \dpc(\M,\N; \C) + \dpc(\N,\L; \C').
$$
Taking the infimum over $\C \in \A(\M,\N)$ and $\C' \in \A(\N,\L)$ we
get
\begin{align*}
\dcp(\M,\L)
\leq \inf_{\C\in\A(\M,\N)} \dpc(\M,\N ; \C) + \inf_{\C'\in \A(\N,\L)}
\dpc(\N,\L; \C')  = \dcp(\M,\N) + \dcp(\N,\L).
\end{align*}
\end{proof}

We conclude this section by providing an approximation result: given
two surfaces $\M, \, \N$ and a correspondence map $\C\in\A(\M,\N)$,
we would like to approximate the centroids $\int_\M x\ d\vol_\M(x),
\int_{\N} y \ d\vol_{\N}(y)$ and the integral defined in Proposition
\ref{prop:rotational_part}; these approximations will be used to
compute approximations to the optimal rigid transformations and to the
distances $\dpc(\M,\N;\C)$. To this end we will
use a simple rectangle-type integration formula that we describe
now. Let $Q=\set{q_\ell}_{\ell=1}^L \subset \M$ be a set of points
such that their corresponding Voronoi cells
$\set{\Upsilon_\ell}_\ell$ have approximately equal surface area; in
practice, such a set of points can be determined by means of the
Furthest Point Algorithm (FPS) \cite{ eldar97farthest}. Using the
notation $\vol_\M(\Upsilon_\ell) = \int_{\Upsilon_\ell} d\vol_\M $
we then have
\begin{equation}\label{e:rectangle_rule_on_surfaces}
    \int_\M f(x) d\vol_\M(x) \approx \sum_{\ell} f(q_{\ell})\vol_\M(\Upsilon_\ell).
\end{equation}
The error made in this approximation can be estimated in terms
of the \emph{fill distance} $\eta_\M(Q)$ of the set $Q$, defined as
\begin{equation}\label{e:fill_distance}
    \eta_\M(Q):=\sup \Big\{ r\in\Real \ \Big | \
\exists\, x\in\\M \ s.t. \ B_\M( x,r)\cap Q = \emptyset    \Big \},
\end{equation}
where $B_\M(x,r) = \{q\in\M \ | \ \gd_{\M}(x,q)<r\}$, with
$\gd_{\M}(x,q)$ the geodesic distance on $\M$ between $x$ and
$q$. Intuitively, the fill distance $\eta_\M(Q)$ is the radius of
the largest geodesic open ball that can be placed on the surface
$\M$ without including any point of the (discrete) set $Q$. In other
words it is the largest ``circular hole'' in the sampling $Q\subset
\M$. We have then
\begin{prop}
The error of the approximation {\rm (\ref{e:rectangle_rule_on_surfaces})} has the following upper bound:
\begin{align*}
\abs{\int_\M f(x) d\vol_\M(x) - \sum_{\ell}
f(q_\ell)\vol_\M(\Upsilon_\ell)}
\leq
 \mathop{\sup}_{\ell,x\in \Upsilon_\ell}\abs{f(x)-f(q_\ell)}
\leq M \eta_\M(Q),
\end{align*}
where $M$ is a bound
on the norm of the gradient $\nabla_\M f$ of $f$.
Hence the error is linear in the separation distance.
\end{prop}
\begin{proof}
Writing $\M=\cup_\ell \Upsilon_\ell$ we get
\begin{align*} \abs{\int_\M f(x) d\vol_\M(x) - \sum_{\ell} f(q_\ell)\vol_\M(\Upsilon_\ell)}
 & \leq
 \sum_\ell \int_{\Upsilon_\ell}\abs{f(x)-f(q_\ell)}d\vol_\M(x) \\
 & \leq
 \mathop{\sup}_{\ell,x\in \Upsilon_\ell}\abs{f(x)-f(q_\ell)} \sum_\ell \vol_\M(\Upsilon_\ell) \\
 & = \mathop{\sup}_{\ell,x\in \Upsilon_\ell}\abs{f(x)-f(q_\ell)},
 \end{align*}
 where the last equality uses $\vol_\M(\M)=1$. Now take arbitrary $x\in\Upsilon_\ell$,
 and denote by $\gamma(t):[0 , d_\M(q_\ell,x)]\too \M$ the arc-length speed geodesic curve
 connecting $q_\ell$ and $x$. Then,
 $$f(x)-f(q_\ell) = \int_0^{d_\M(q_\ell,x)}\frac{d}{dt}\brac{f(\gamma(t))}dt =
 \int_0^{d_\M(q_\ell,x)}\ip{ \nabla_\M f(\gamma(t)), \dot{\gamma}(t)}_\M dt.$$
Using the Cauchy-Schwarz inequality,
\begin{align*}
\abs{f(x)-f(q_\ell)} & \leq \mathop{\sup}_{x\in \M} \|\nabla_\M f(x)\|_\M \int_0^{d_\M(q_\ell,x)} \norm{\dot{\gamma}(t)}_\M dt \\
& = \mathop{\sup}_{x\in \M} \|\nabla_\M f(x)\|_\M \ d_\M(q_\ell,x).
\end{align*}
 Lastly, the inequality $d_\M(q_\ell,x) \leq \eta(\M)$ can be derived directly from the properties
 of Voronoi cells (see Lemma D.2 in \cite{Lipman_Puente_Daubechies:2010:computational}).
\end{proof}


\section{M\"{o}bius transformations as a reduced search space}

Computing the surface Procrustes distance, as we defined it above,
amounts to solving a hard optimization problem: unfortunately, the
sets $\A(\M,\N)$ or $\mathcal{B}_B(\M,\N)$
are formally
infinite dimensional manifolds, and
therefore extremely hard to search in practice. Our key idea is to
replace the search space $\A(\M,\N)$ in the variational formulation
(\ref{e:procrustes_and_correspondence}) by another, much smaller,
set of maps. The core observation is that the set of conformal (or
anti-conformal) mappings
between $\M$ and $\N$, which has a
finite (and small) dimensionality, gets ``close'' (in some
sense to be made precise below) to the
minimizing $\C\in \A(\M,\N)$. In particular, we shall see that if $\dcp (\M,\N)$ is small,
then the minimizing area-preserving map $\C$ in (\ref{e:procrustes_and_correspondence})
is close to conformal.

Let us explain this in some more detail.
We are particularly interested in computing (approximate) continuous Procrustes
distances for ``close'' pairs \cite{pnas}. In those cases the insight
that (close to) optimal $\C$ have to be close to conformal
leads us to a strategy that involves
minimizing over a much smaller set of maps.
To achieve this, we shall make use of a nonlinear
procedure $\pr$ that ``transforms'' a map that is close
to $\A(\M,\N)$ into an area-preserving map, i.e. to an element of $\A(\M,\N)$.
This nonlinear transformation leaves elements of $\A(\M,\N)$ unchanged, and can thus be
interpreted as a nonlinear projection procedure (hence the notation).
The smaller set of maps over which we shall minimize is then the image in $\A(\M,\N)$ of
the family of conformal maps (from $\M$ to $\N$),
transformed by  $\pr$.

As a search space, the family of conformal mappings is a much more
``friendly'' setting than $\A(\M,\N)$ or $\mathcal{B}_B(\M,\N)$.
First, by the uniformization theorem the
conformal (or anti-conformal) bijective mappings
can be characterized completely, and an explicit parameterization can be given in terms of a small
number of parameters. For instance, the family of conformal bijective mappings
between two disk-type surfaces $\M, \, \N$ is represented by
(disk-preserving) M\"{o}bius transformations. Each mapping in this
family is completely characterized by 3 real (bounded) parameters;
therefore the search over the space of conformal mappings can be
done efficiently. Second, M\"{o}bius transformations are smooth
bijective diffeomorphisms, so that our candidate search space
consists of only ``nice'' intrinsic mappings.

To motivate why we would consider restricting ourselves to conformal
mappings (or their deformations through $\pr$) for the optimization,
we note that for $\M, \, \N$ such that $\dcp(\M,\N)=0$, the infimum in
(\ref{e:procrustes_and_correspondence}) is achieved for some
$R\in\R$ (by Theorem \ref{thm:continuous_proscrustes_is_a_metric});
in this case the minimizing $\C=R$ is obviously conformal. However,
we prove below the stronger result that
a correspondence
$\C:\M\too\N$ for which the distance $\dpc(\M,\N;\C)$ is small can be
approximated (under rather mild assumptions on the regularity of
$\C$) by a bijective globally conformal mapping from $\M$ to $\N$.

We start with a few simple
lemmas. The first Lemma is proved in \cite{Smolyanov00}:
\begin{lem}\label{lem:d_g_x_y_minus_norm_x_y}
Let $\M\subset \Real^3$ be a compact {\rm 2}-manifold with the
induced Riemannian metric $g$. Then
$$\Big | \gd_\M(x,x') - \norm{x-x'} \Big | \leq C_{\M}  \gd_\M(x,x')^{3},$$
where $\gd_\M(x,x')$ denotes the geodesic distance between $x$ and
$x'$, $\norm{x-x'}$ denotes the Euclidean distance between these
points, and $C_{\M}$ depends only on the curvature of $\M$.
\end{lem}

Next, we prove a result concerning the approximation of the norm of the differential of a map:
\begin{lem}\label{lem:div_diff_approx_norm_differential}
Let $\M, \, \N\subset \Real^3$ be compact {\rm 2}-manifolds with the
induced Riemannian metrics $g,h$ (respectively). Let $F:\M\too\N$ be
a smooth map, and denote by $DF_x$ the differential of $F$ at
arbitrary $x\in\M$. Then for $\delta >0$ sufficiently small, the
following holds: for all $x\in\M$, there exists $x'$ in the boundary
$\partial B_g(x,\delta)$ of the $\delta$-radius geodesic ball
centered at $x$, $\partial B_g(x,\delta):=\{u\in \M|
\,\gd_\M(x,u)=\delta \}$, such that
$$
\Bigg | \frac{\norm{F(x)-F(x')}}{\norm{x-x'}} - \norm{DF_x}_{g,h}
\Bigg | \leq \widetilde{C}_\kappa\,\delta,
$$
where $\norm{\cdot}_{g,h}$ is the operator norm associated with the
norms $\norm{\cdot}_g$ and $\norm{\cdot}_h$ in the usual way, i.e.
$\norm{L}_{g,h}= \sup_{\xi\ne 0}\frac{\norm{L
\xi}_h}{\norm{\xi}_g}$. Here, $\widetilde{C}_\kappa$ depends only on
the maximum of the surfaces' curvature and norms of second order
differentials of the mapping $F$.
\end{lem}
\begin{proof}
For $\xi \in \Real^2$, we set $\D(\delta)=\set{\xi \in \Real^2 \
\mid \ \norm{\xi}\leq\delta}$. Fix $x$ and take $\delta>0$ small
enough so that the exponential map $exp_x:\D(2\delta)\too
B_g(x,2\delta)$ is a diffeomorphism. For every $x'\in
B_g(x,\delta)$, we denote by the vector $\xi_{x'}$ the vector in $
\D(\delta)$ such that $exp_x(\xi_{x'})=x'$; in other words,
$\xi_{x'}$ is tangent to the geodesic on $\M$ that goes from $x$ to
$x'$, and  $\norm{\xi_{x'}} = \gd_\M(x,x')$. Denote $\wt{F}=F\circ
exp_x : \D(\delta) \too \N$, and consider the line  $\gamma(t) =
t\,\xi_{x'}$, $0\leq t \leq 1$. Then
\begin{align}
\label{e:Cx_minus_Cx_tag} \nonumber F(x')-F(x) & = \int_0^1
\frac{d}{dt} \brac{\wt{F}(\gamma(t))  } dt = \int_0^1
D\wt{F}_{\gamma(t)} \dot{\gamma}(t) dt \\  & =\int_0^1 \Big (
D\wt{F}_{\gamma(0)} + O(\norm{\xi_{x'}}) \Big ) \xi_{x'} dt  =
D\wt{F}_x \xi_{x'} + O(\gd_\M(x,x')^2)\,.
\end{align}

Let $\xi\in\partial \D(\delta)$ be such that
$\norm{D\wt{F}_x(\xi)}_h = \norm{D\wt{F}_x}_{g,h} \norm{\xi}_g$.
Remember that at $exp^{-1}(x)$ the pull-back metric tensor
$(exp_x^*g)$ equals $\delta_{ij}$, and therefore $\norm{\xi}_g =
\norm{\xi}$. Also note that $\norm{D\wt{F}_x(\xi)}_h =
\norm{D\wt{F}_x(\xi)}$, since the metric $h$ is induced by the
ambient Euclidean metric of $\Real^3$. Now set $x'=exp_x(\xi)$,
and take the Euclidean norm of both sides of
(\ref{e:Cx_minus_Cx_tag}). Then we have
$$\norm{F(x')-F(x)} =  \norm{D\wt{F}_x}_{g,h} \gd_\M(x,x') + O(\gd_\M(x,x')^2),$$
and therefore
$$
\frac{\norm{F(x')-F(x)}}{\gd_\M(x,x')} =  \norm{D\wt{F}_x}_{g,h} +
O(\gd_\M(x,x')).
$$
Using Lemma \ref{lem:d_g_x_y_minus_norm_x_y}, this leads to the
desired estimate.
\end{proof}

Next, we define the \emph{cone condition} for a surface $\M$:
\begin{defn}\label{def:cone_cond_M}
We say that a compact 2-manifold $\M \subset \Real^3$ satisfies the
$(\sigma,\theta)$-\emph{cone condition}, where $\sigma>0$ and $\theta\in (0,2\pi]$,
if for every
$x\in\M$ there is a unit vector $\mathbf{n}$ in the tangent
plane $T_x\M$ such that the exponential map $exp_x$ is well-defined
on the
cone
$c_\M(\sigma,\theta;\mathbf{n})=\set{\xi
\in \Real^2 ; \norm{\xi}\leq\sigma \, , \,
\ip{\xi,\mathbf{n}}\geq \norm{\xi}\,\cos(\theta/2)}$,
and is one-to-one on the whole cone.
\end{defn}
In other words, the surface $\M$ satisfies the $(\sigma,\theta)$-cone condition if for every $x\in\M$,
there is a ``fan'', spanning at least an angle $\theta$, of geodesics that
leave $x$ and continue, within $\M$, for at least a distance $\sigma$ (w.r.t. the
metric induced on $\M$ by $\Real^3$), without intersecting themselves or any other geodesic in the fan.
We have now

\begin{lem}\label{lem:ingrid_condition}
Let $\M\subset \Real^3$ be a compact {\rm 2}-manifold
satisfying the ($\sigma,\theta)$-cone
condition. Then there exist
constants $\rho\ ,\, \Gamma>0$ depending on $\sigma$, $\theta$
and on the curvature $\kappa$ of $\M$
such that for all $u\in\M$ and all $r < \rho$, the area of $\set{x
\in \M|\,\norm{u-x}\leq r }$ is bounded below by $\Gamma r^2$
(with $\norm{\cdot}$ standing for the Euclidean norm in $\Real^3$).
\end{lem}
\begin{proof}
By Lemma \ref{lem:d_g_x_y_minus_norm_x_y} there exists a constant
$R>0$ (depending only on the curvature of $\M$) such that for all
$x,y\in \M$ satisfying $\gd_\M(x,y) < R$ we have $$\gd_\M(x,y) >
\frac{1}{2} \norm{x-y}.$$

Set $\rho_0 = 2\,\min \set{R,\sigma}$, and fix an arbitrary
$u\in\M$. We have then, for all $r<\rho_0$, that $\set{x \in
\M|\,\gd_\M(u,x)\leq r/2 } \subset \set{x \in \M|\,\norm{u-x}\leq r
}$, and consequently
$$\int_{\set{x \in \M|\,\norm{u-x}\leq r }}d\vol_\M(x) \geq
\int_{\set{x \in \M|\,\gd_\M(u,x)\leq r/2 }}d\vol_\M(x).$$

Now introduce polar coordinates $(\tau, \phi)$ on the tangent plane $T_u\M$, so that
the vector $\mathbf{n}$ (with respect to which the cone condition holds at $u$)
is aligned with the direction $\phi =0$. With respect to this
coordinate system, the exponential $exp_u$ maps $[0,\sigma]\times[-\theta/2,\theta/2]$
to $\M$, and
the metric density can be written as
\cite{spivak1999comprehensive_vol2}:
$$
\sqrt{g}(\tau,\phi) = \tau -
\tau^3\frac{\kappa(u)}{6} + o(\tau^3).
$$
Since $r < \rho_0 \leq
2\sigma$, the sector
$[0,r/2]\times[-\theta/2,\theta/2]$ is contained in
$[0,\sigma]\times[-\theta/2,\theta/2]$ and we have
\begin{align*}
\int_{\set{x \in \M|\,\gd_\M(x,y)\leq r/2 }}d\vol_\M(x) & \geq
\int_{-\theta/2}^{\theta/2}\,\int_{0}^{r/2}\sqrt{g}(\tau,\phi)\,d\tau \,
d\phi \\ & = r^2\frac{\theta}{8} - r^4 \frac{\kappa(u) \theta}{384} +
o(r^4)\, = \, r^2\frac{\theta}{8}\,\Big(1+\,O(r^2)\Big)\,,
\end{align*}
where, as usual, the absolute value of the $O(r^2)$ term is bounded above by $Cr^2$ , 
for some $C>0$, for all
$r$ smaller than some $r_1$.
Setting $\rho=\min(\rho_0, r_1, 1/\sqrt{2C})$ and $\Gamma = [1- \min(1/2,C\,r_1^2)]\,\theta/8 $ 
we obtain, for $r< \rho$,
$$
\int_{\set{x \in \M|\,\gd_\M(x,y)\leq r/2 }}d\vol_\M(x)  \geq \Gamma
\,r^2\,,
$$
completing the proof.
\end{proof}

We are ready to prove the main result of this section, which provides a bound on the
conformal distortion $dis_{\C}$ of the optimal area-preserving
``alignments'' $\C$ for surfaces $\M$ and $\N$ that are close to
each other in the continuous Procrustes distance. The conformal
distortion $dis_{\C}(x)$ of $\C$ at $x$ is defined as the ratio
between the two singular values of the matrix obtained by expressing
the differential $D\C_x$ with respect to orthonormal bases in
$T_x\M$ and $T_{\C x}\N$, respectively.
\begin{thm}\label{thm:small_procrustes_dist_means_almost_conformal}
Let $\M, \, \N\subset \Real^3$ be  {\rm 2}-manifolds with induced
Riemannian metrics $g,\,h$ (respectively), with curvatures bounded
above by $\kappa$, and satisfying the $(\sigma,\theta)$-cone condition.
We consider area-preserving diffeomorphisms $\C:\M\too\N$ with first
and second order differentials bounded by $M$.
Then, for sufficiently small $\epsilon$, the bound $\dpc(\M,\N;\C)
\leq \epsilon$ implies the following bound on the conformal
distortion $dis_{\C}$ of the map $\C$:
$$
\sup_{x \in \M} dis_{\C}(x) \leq 1 + O(\epsilon^{1/4}),
$$
where the constant in the $O$-notation depends on only $\kappa$ and
$M$.
\end{thm}
\begin{proof} Denote by $R\in\R$ the rigid motion for
which the infimum in
(\ref{eq:dpc}) is attained for $\C$ .

The first step in our proof is to derive a uniform bound on $\norm{R(x)-\C(x)}$.
We start by noting that the function $q(x) = \norm{R(x)-\C(x)}$ is
Lipschitz with a constant $\lambda$ dependent only on
$M$.
Indeed, we have
\begin{align*}
\Big | \norm{R(x)-\C(x)} - \norm{R(y)-\C(y)} \Big | & \leq
\norm{R(x)-R(y)} + \norm{\C(x)-\C(y)} \\ &
= \norm{x-y} + \norm{\C(x)-\C(y)} \,.
\end{align*}
By assumption, $ \gd_\N(\C(x),\C(y)) \leq M \gd_\M(x,y)$. By Lemma
3.1, $\gd_\M(x,y)\leq 3/2\, \norm{x-y}$ if $\gd_\M(x,y) $ is
sufficiently small. On the other hand, we have, for all $x'$, $y' \in
\N$, $\norm{x'-y'} \leq \gd_\N(x',y')$,
$\gd_\N$ is the metric induced on $\N$ by the Euclidean metric in
$\Real^3$. Thus $ \norm{\C(x)-\C(y)} \leq 3M/2 \norm{x-y}$ when
$\norm{x-y} $ is sufficiently small. Since on the other hand $\N$
is compact and thus bounded,  $ \norm{\C(x)-\C(y)}$ is bounded
uniformly in $x$, $y$, regardless of $\norm{x-y}$. It follows that
there exists a constant $\lambda$, depending only on the geometric
properties of the surfaces $\M$ and $\N$, and on $M$, such that
$$
\Big | \norm{R(x)-\C(x)} - \norm{R(y)-\C(y)} \Big |  \leq\lambda \norm{x-y}\,.
$$
Suppose $q$ attains its maximum $Q$ in $u\in \M$.
Set $\alpha = \min \Big(\rho, \frac{Q}{2\lambda}\Big)$, with $\rho>0$
as in Lemma
\ref{lem:ingrid_condition}.
Then we must have
\begin{align*}
\int_{\set{x\in\M|\,\norm{u-x}<\alpha}}\, \max(0, Q-\lambda
\norm{u-x})^2\,\,d\vol_\M(x) &\leq
\int_{\set{x\in\M|\,\norm{u-x}<\alpha}}\,  \norm{R(x)-\C(x)}^2
\\ &\leq \dpc(\M,\N;\C)^2\leq \epsilon^2\,.
\end{align*}

On the other hand, we also have, by Lemma
\ref{lem:ingrid_condition}, and using $Q-\lambda\norm{x-u} \geq Q-\lambda \alpha$
on $ \set{x\in\M|\,\norm{u-x}<\alpha}$,
\begin{align*}
\int_{\set{x\in\M|\,\norm{u-x}<\alpha}}\, &\max(0, Q-\lambda
\norm{u-x})^2\,\,d\vol_\M(x)  \geq
\int_{\set{x\in\M|\,\norm{u-x}<\alpha}}\,
(Q-\lambda\alpha)^2\,\,d\vol_\M(x) \\
&= \,(Q-\lambda\alpha)^2\,
\int_{\set{x\in\M|\,\norm{u-x}<\alpha}}\,d\vol_\M(x)
\, \geq \,\frac{Q^2}{4} \,\Gamma \alpha^2
=\, \frac{\Gamma}{4}\, Q^2 \min \Big(\rho, \frac{Q}{2\lambda}\Big)^2.
\end{align*}
This implies, in particular, that
$$
\frac{\Gamma}{4}\, Q^2 \min \Big(\rho, \frac{Q}{2\lambda}\Big)^2\leq \, \epsilon^2 \,.
$$
If $Q/(2 \lambda) > \rho $, then it follows that $\Gamma \,Q^2\,\rho^2/4 < \epsilon^2$,
hence (by using $Q/(2 \lambda) > \rho $ once again) $\Gamma \lambda^2 \rho^4 <\epsilon^2$. Note that $\Gamma$, $\lambda$ and $\rho$ are constants that depend on only the geometrical bounds that we impose on $\M$, $\N$
separately; {\em a priori} they bear no relationship to whether or not the continuous Procrustes distance between the surfaces is small. With a left hand side independent
of $\epsilon$ and strictly positive, the inequality above can therefore not
be satisfied  if $\epsilon$ is sufficiently small; more precisely, if
$\epsilon \leq \Gamma^{1/2}\, \lambda \,\rho^2  $), then this case is excluded.

For sufficiently small $\epsilon$, we have thus $Q/(2 \lambda) \leq \rho $, implying
$\Gamma \,Q^4/(16 \lambda^2) \leq \epsilon^2$, or $Q \leq 2\,\lambda^{1/2}\,\Gamma ^{-1/4}\,
\epsilon^{1/2}$. In other words,
there exists a constant $C_1>0$, dependent on only
$\lambda$, $\rho$, $M$ and $\kappa$, such that, for
sufficiently small $\epsilon$,
$$\mathop{\max}_{x\in\M}\norm{R(x) - C(x)} = Q \leq C_1 \,\epsilon^{1/2}\,,$$
which is the desired uniform bound on $\norm{R(x) - C(x)} $.

Second, by Lemma \ref{lem:div_diff_approx_norm_differential} we can
take $y\in \partial B_g(x,\epsilon^{1/4})$ such that
$$\frac{\norm{\C(y)-\C(x)}}{\norm{x-y}} = \norm{D \C_x}_{g,h} + O(\epsilon^{1/4}).$$
Using the triangle inequality as well as
$\norm{R(x)-R(y)}=\norm{x-y}$, and applying Lemma
\ref{lem:d_g_x_y_minus_norm_x_y}, we obtain
$$\frac{\norm{\C(y)-\C(x)}}{\norm{x-y}} \leq
\frac{\norm{\C(y)-R(y)}+\norm{x-y}+\norm{R(x)-\C(x)}}{\norm{x-y}}\leq
1 + O(\epsilon^{1/4}),$$ and thus
$$ \norm{D \C_x}_{g,h} \leq 1 + O(\epsilon^{1/4}).$$
Lastly, since $\norm{D\C_x}_{g,h}$ equals the larger singular value
of the matrix for $D\C_x$ w.r.t. orthonormal bases of $T_x\M$ and
$T_{\C(x)}\N$ (respectively), and since $\C$ is area-preserving
(implying that the determinant of this $2 \times 2$ matrix equals 1)
the conformal distortion of $\C$ at $x$ is $\norm{D \C_x}_{g,h}^2$,
and thus
$$dis_\C(x) \leq 1 + O(\epsilon^{1/4}).$$
\end{proof}
%
%
%
{ Theorem \ref{thm:small_procrustes_dist_means_almost_conformal}
tells us that area-preserving diffeomorphisms associated to small
surface Procrustes distances have small conformal distortion
everywhere. We will next use the theory of quasi-conformal (QC) maps
to see that, for disk-type surfaces, this implies that such maps then must be ``close'' to conformal maps.

For the sake of convenience, we restrict our discussion here to the case of disk-type surfaces here
(similar results
can be shown for sphere-type surfaces).
More precisely, we start with two disk type surfaces $\M, \N\subset \Real^3$ with
induced metric tensors $g,\,h$ (respectively),
and we consider a global conformal parametrization
(uniformization) of each onto their canonical domain,
$\Psi:\M \too \D$, $\Psi':\N \too \D$.
The surfaces are then
intrinsically represented by their conformal factors $\mu$ and
$\nu$. In other words, the push-forward metric tensors of $\M, \, \N$
under the maps $\Psi$, $\Psi'$ are given by  $(\Psi_* g)[z] =
\mu(z) dz d\bar{z}$, and $(\Psi'_* h)[w] = \nu(w) dw d\bar{w}$,
respectively. The conformal factors also act as ``density
functions'' in the sense that the area in $\M$ of an arbitrary Borel
set $\Omega\subset \M$ can be written as $\vol_\M(\Omega) =
\int_{\Psi(\Omega)} \mu(z) dx dy$, where $z=x+\bfi y$; similarly
for the surface $\N$.

Now every conformal mapping from $\M $ to $\N$ can be written as
$\Psi'^{-1} \circ m \circ \Psi$, where $m$ ranges over the
M\"{o}bius transformations of the unit disk that preserve its
boundary:
\begin{equation}\label{e:disk_mobius}
    m(z)=e^{\bfi \theta }\frac{z-a}{1-z\bar{a}},
\end{equation}
where $\theta \in [0,2\pi)$, $a\in \D$. This family of
transformations has three degrees of freedom (one for the angle and
two for the complex number $a$); we denote the family by $\Mob(\D)$.

Likewise an area preserving (and orientation preserving) map $\C$ from $\M$ to $\N$
can be ``transported'' to $\D$ by means of
$\Psi$ and $\Psi'$, leading us to consider instead $\C_{\mbox{\tiny{tr}}}:=\Psi'\circ\C \circ \Psi^{-1}$,
mapping $\D$ to itself.
We will use QC theory to show that, if $\dpc(\M,\N;\C)$
is small, then $\C_{\mbox{\tiny{tr}}}$ is close to an element of $\Mob(\D)$,
with respect to the maximum norm over the unit disk $\D=\set{z \in \CC\ \mid \
\abs{z}<1}$, at least if $\C$ is
orientation preserving.  If it is orientation
reversing, it is close to an anti-conformal map. We provide
details below for the orientation preserving case; the reversing case is entirely similar.

By an appropriate M\"{o}bius change of coordinates $\wt{m}$,
replacing $\Psi$ by $\wt{\Psi'} = \wt{m}\circ \Psi'$, we
can even ensure
that $\wt{\C_{\mbox{\tiny{tr}}}}:= \wt{m}\circ\C_{\mbox{\tiny{tr}}}$
has $0$ and $1$ as fixed points. Abusing notation, and denoting
$\wt{\C_{\mbox{\tiny{tr}}}}$ by $\C$ again, we thus assume
$\C(0)=0$, and $\C(1)=1$. We shall show that $\C$ is close to the identity, which
means that $\C_{\mbox{\tiny{tr}}} $ is close to $\wt{m}^{-1}$, and thus that
the original area preserving map from $\M$ to $\N$ is close to the conformal
map from $\M$ to $\N$ given by $ (\Psi')^{-1}\circ \wt{m}^{-1} \circ \Psi$.

We consider, as is very customary in complex analysis, derivatives with respect
to $z$ and $\bbar{z}$ of
the differentiable map $\C$ from the subset $\D$ of $\CC$ to itself, i.e.
$$
\frac{\partial \C}{\partial z}= \frac{\partial \C}{\partial x}\,-\,i\,\frac{\partial \C}{\partial y}
\quad\mbox{ and } \quad
\frac{\partial \C}{\partial \bbar{z}}= \frac{\partial \C}{\partial x}\,+\,i\,\frac{\partial \C}{\partial y}
\,\,,\quad\mbox{  where }\quad z=x+iy.
$$
We define
the {\em complex dilation} of $\C$ by
$$
\varrho
= \frac{\partial \C}{\partial \bbar{z}}\mbox{\Huge{/}}\frac{\partial \C}{\partial z}.
$$
For orientation preserving $\C$ we have (see
for example, \cite{ahlfors1966lectures}):
$$
dis_{\C} = \frac{1+\abs{\varrho}}{1-\abs{\varrho}} \ , \ \abs{\varrho} = \frac{dis_\C -1}{dis_\C +1}.
$$
Theorem \ref{thm:small_procrustes_dist_means_almost_conformal}
therefore implies, uniformly on $\D$,
\begin{equation}
\label{rho_order}
\varrho = O(\epsilon^{1/4}).
\end{equation}

We will use the following existence and uniqueness theorem for the
Beltrami equation (see \cite{imayoshi1992introduction}, Theorem
4.30):
\begin{thm}\label{thm:exitence_sol_Beltrami_eq}
For every $\varrho:\CC\too \CC$ measurable such that
$\norm{\varrho}_\infty < 1$, there exists a homeomorphism $f$ of
$\CC$ onto $\CC$ which is a quasiconformal mapping of $\CC$ with
complex dilation $\varrho$. Moreover, $f$ is uniquely determined by
the following normalization conditions: $f(0)=0$, $f(1)=1$, and
$f(\infty)=\infty$.
\end{thm}
As is customary, we will call \emph{normalized solution} any solution of a Beltrami
equation that satisfies the normalization conditions.

Theorem \ref{thm:exitence_sol_Beltrami_eq} requires the complex dilation $\rho$ to be defined
on all of $\CC$. Before applying it, we thus need to first obtain  $\rho$ on all of $\CC$, which
we do by
extending $\C$ from $\D$ to the entire complex plane $\CC$ by reflection:
$$\wh{\C}(z) = \left \{ \begin{array}{cc}
                            \C(z) & \abs{z}\leq 1 \\
                            1\Big/\bbar{\C(1/\bbar{z})} & \abs{z}>1
                          \end{array}
\right .
$$
Note that the extension $\wh{\C}$
is normalized, that is, it satisfies $\wh{\C}(0)=0,\, \wh{\C}(1)=1, \,
\wh{\C}(\infty)=\infty$. Moreover, this extension preserves the
conformal distortion, that is, for $\abs{z}>1$:
$$dis_{\wh{\C}}(z) = dis_\C(1/\bbar{z}).$$
It follows that this extension of $dis_{\C}$ to all of $\CC$ still satisfies (\ref{rho_order}).
We now have

\begin{lem}\label{lem:barC_satisfy_beltrami}
The extension $\wh{\C}:\CC\too\CC$ is the unique normalized solution to the
following Beltrami equation:
$$\wh{C}_{\bbar{z}}(z) = \wh{\varrho}(z) \, \wh{\C}_z(z),$$
where $\wh{\varrho}$ is a complex dilation (a.k.a. Beltrami
coefficient) defined by
$$\wh{\varrho}(z) = \left \{ \begin{array}{cc}
                            \varrho(z) & \abs{z}<1\\
                            0 & \abs{z}=1 \\
                            \bbar{\varrho(1/\bbar{z})} \parr{\frac{z^2}{\bbar{z}^2}} & z \in \abs{z}>1
                          \end{array}
  \right .$$
\end{lem}
\begin{proof}
A straightforward calculation shows that $\wh{\C}$ has the Beltrami
coefficient $\wh{\varrho}$ almost everywhere. As mentioned above,
$\wh{\C}$ satisfies the normalization conditions of Theorem
\ref{thm:exitence_sol_Beltrami_eq} and therefore the uniqueness follows from 
Theorem
\ref{thm:exitence_sol_Beltrami_eq}.
\end{proof}

We will next use Proposition 4.36 from
\cite{imayoshi1992introduction}, the statement of which is:
\begin{thm}\label{thm:dependence_on_Beltrami_coef}
If $\norm{\varrho}_\infty \too 0$, then the normalized solution of
the Beltrami equation $f^\varrho$ converges to the identity in the
maximum norm on $\D$, $\norm{f^\varrho-\mbox{\it{Id}}}_\infty \too 0$, where
$\mbox{\it{Id}}(z)=z$.
\end{thm}
The proof of Proposition 4.36 in \cite{imayoshi1992introduction}  actually
demonstrates a slightly stronger claim:
\begin{thm}\label{thm:dependence_on_Beltrami_coef_2}
If $\norm{\varrho}_\infty \too 0$, then the normalized solution of
the Beltrami equation $f^\varrho$ satisfies
$$\norm{f^\varrho-\mbox{\it{Id}}}_\infty \leq M \norm{\varrho}_\infty, $$ on
$\D$, for some constant $M>0$ independent of sufficiently small
$\varrho$.
\end{thm}

Combining Theorems
\ref{thm:small_procrustes_dist_means_almost_conformal} and
\ref{thm:dependence_on_Beltrami_coef_2} finally yields
\begin{thm}\label{thm:MAIN_RESULT_low_cpd_imlpies_close_to_conformal}
Let $\M, \, \N\subset \Real^3$ be  {\rm 2}-manifolds with induced
Riemannian metrics $g,\,h$ (respectively), with curvatures bounded
above by $\kappa$, and satisfying the $(\sigma,\theta)$-cone
condition. We consider area-preserving and orientation-preserving
diffeomorphisms $\C:\M\too\N$ with first and second order
differentials bounded by $M$. Let $\Psi:\M\too\D$, $\Psi':\N\too\D$
be uniformizing maps of $\M, \, \N$ onto the disk. Let $m$ be a
disk-preserving M\"{o}bius transformation such that $f=m\circ \Psi'
\circ \C \circ \Psi^{-1}:\D\too\D$ satisfies $f(0)=0,f(1)=1$.
Then the bound $\dpc(\M,\N;\C) \leq \epsilon$ implies the following
bound:
$$
\norm{f - \mbox{\it{Id}}}_\infty = O(\epsilon^{1/4}),
$$
where $\mbox{\it{Id}}(z)=z$ is the identity map, and where the constant in the
$O$-notation depends on only $\kappa$, $\sigma$, $\theta$ and $\M$.
\end{thm}

The orientation reversing $\C:\M\too\N$ are close to the
anti-M\"{o}bius transformations that can be calculated from the
M\"{o}bius transformations by setting
\begin{equation}\label{e:disk_anti_mobius}
    \overline{m}(z)=m(\bar{z}),
\end{equation}
where $m$ is any M\"{o}bius transformation.

As described earlier, we use this theorem as a guide to build an efficient search algorithm
to compute (an approximation to) $\dpc(\M,\N)$ for surfaces $\M$, $\N$ that are not hugely dissimilar.
Since area-preserving maps $\C$ from $\M$ to $\N$ that are close to minimizing
$\dpc(\M,\N;\C)$ must be close to conformal, we
start by searching $\Mob(\M,\N)$ to find the conformal or
anti-conformal map $m$ that minimizes $\dpc(\M,\N;m)$. We then
transform this $m$ into a nearby
area-preserving diffeomorphism by means of a nonlinear
transform $\pr$, still to be defined below. We expect (but do no prove)
that $\pr(m)$ is then a good approximation to (nearly) minimizing $\C$.
Note that there are no guarantees that
this approximation process, in which we replace
$\A(\M,\N)$ by the proxy $\pr\parr{\Mob(\M,\N)}$, preserves
the triangle inequality property of
$\dcp(\M,\N)$; the approximations we compute therefore result in a
measure of dissimilarity rather then a distance.

}

\section{Searching appropriate M\"{o}bius candidates and massaging them into area-preservation}

In the previous section we showed that it is useful
to first find a M\"{o}bius transformation $m$ for which
$\dpc(\M,\N;m)$ is small; we show in subsection \ref{subsec_4.1} below a practical strategy
for obtaining such candidate M\"{o}bius transformations that is fast and
efficient for our applications. To obtain a better approximation of
the optimal element of $\A(\M,\N)$ from these candidate
$m\in\Mob(\M,\N)$, we will, in subsections \ref{subsec_4.2} and  \ref{subsec_4.3}, 
deform each of them into a nearby
area-preserving map. That is, for every $m\in\Mob(\M,\N)$, we will
construct a map $f_m:\M\too\N$ such that $f_m \circ m \in
\A(\M,\N)$.

\subsection{Searching the M\"{o}bius group.}
\label{subsec_4.1}
By Theorem \ref{thm:MAIN_RESULT_low_cpd_imlpies_close_to_conformal}
we know that an area-preserving diffeomorphism $\C:\M\too \N$ that
produces a small continuous Procrustes distance $d_P(\M,\N;\C)$, is
close to a M\"{o}bius transformation, when written in 
uniformizing coordinates.
\begin{figure}
\includegraphics[width=0.6\columnwidth]{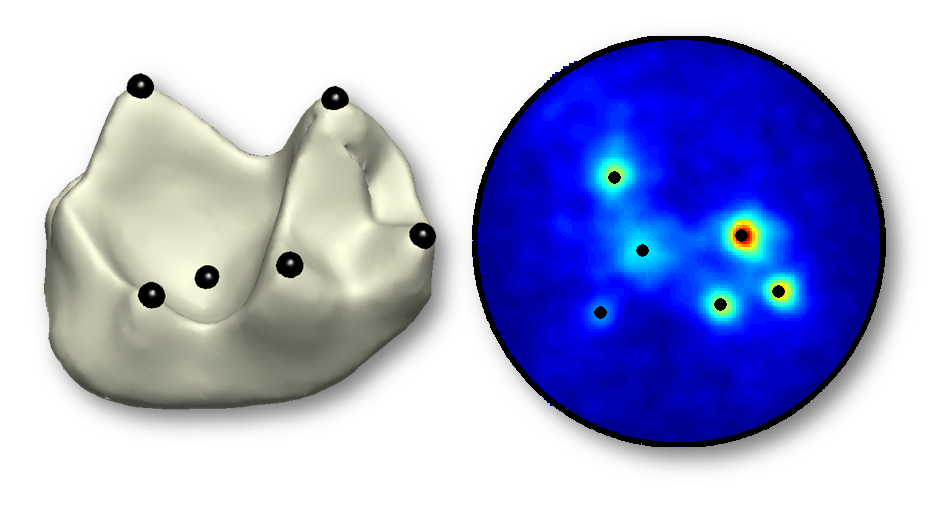}\\
  \caption{A surface (left) and its density function over the uniformization disk.
  The local extremas of the conformal density function are 
  shown as black dots.}\label{fig:density_extrema}
\end{figure}
Hence, we first describe how we search for candidate M\"{o}bius
transformations $m\in\Mob(\M,\N)$ that (we hope) are already close
to area-preserving for our applications. As mentioned above, the
M\"{o}bius group between two disk-type surfaces has three real
degrees of freedom: prescribing the image $w_0\in \D$ of one point
$z_0 \in \D$, as well as one angle $\theta\in[0,2\pi)$, uniquely
defines a disk-preserving M\"{o}bius transformation $m:\D\too\D$. To
speed up the search, we start by determining a mapping for which the
density peaks, i.e. the local extrema of the density $\nu(m(z))$
(more or less) correspond to those of $\mu(z)$. To that end we first
extract, for each surface, a set of extremal points $I_\M,I_\N$
(local maxima and minima) defined by local extrema of the
corresponding density functions $\mu,\nu$, respectively. See Figure
\ref{fig:density_extrema}, where the black points show these
extremal sets. In practice, we find that, in the application (to
bone surfaces) that first motivated us, these points (intimately related
to extrema of Gauss curvature) were likely to contain at least one
pair of corresponding points, across a wide range of examples; this feature has presisted
for other families of examples we examined. Note
that this definition of $I_\M,I_\N$ is not invariant to M\"{o}bius
transformations in the sense that for any M\"{o}bius transformation
the extrema of the pulled-back density $\nu(m(z))\abs{m'(z)}^2$ are
not, in general, the same as the $m^{-1}(w_{\ell})$, where the $w_{\ell}$ are the extrema of $\nu(w)$. 
To make the computation invariant it is sufficient to
search for the extrema of the hyperbolic normalized densities
$(1-\abs{z}^2)^2\mu(z)$ and $(1-\abs{w}^2)^2\nu(w)$ (which are
invariant
to M\"{o}bius change of coordinates).\\
In our algorithm we consider the collection of M\"{o}bius
transformations $m=m(z;\theta,p,q)$ defined by $m(p)=q$ for every
pair $(p,q)\in I_\M \times I_\N$, and every angle $\theta \in
[0,2\pi)$. In order to compute the M\"{o}bius transformations in
practice between two surfaces, we use the algorithm described in
\cite{lipman:2009:mvf,Lipman_Puente_Daubechies:2010:computational}.
Furthermore, we discretize $\theta$:  $\theta = 2\pi\,k/K,
k=0,1,2,...,K-1$. From every candidate M\"{o}bius $m(z;\theta,p,q)$
we build a candidate correspondence map $\C:\M\too \N$ by the steps
described in the next two subsections, deforming it to an
area-preserving $C_m = f_m \circ m$.

One additional remark is that in the above algorithm we also
consider all possible anti-M\"{o}bius transformations
$\wt{m}=\wt{m}(z;\theta,p,q)$ by taking $\wt{m}(z) = m(\bbar{z})$,
where $m$ is a M\"{o}bius transformation, and such that
$\wt{m}(p)=q$ for every pair $(p,q)\in I_\M \times I_\N$, and every
angle $\theta \in [0,2\pi)$.

\subsection{Projection onto $\A(\M,\N)$}
\label{subsec_4.2}
Our goal now is to construct a map $f_m:\M\too\N$ such that $f_m
\circ m \in \A(\M,\N)$, and $f_m$ is (in some sense) ``as close as
possible'' to the identity.

Denote, as before, by $\mu(z),\nu(w)$ the densities of the surfaces
$\M, \, \N$ (resp.) over the unit disk $\D$, defined by $\parr{(m \circ \Psi)_*
g}[z] = \mu(z)\, dz\, d\bbar{z}$, and $\parr{\Psi'_*
h}[w] = \nu(w)\,dw\,d\bbar{w}$. Then, a simple and natural approach
to define $f_m$ is via a ``linear interpolation of the measures'' technique due
to Moser \cite{moser1965}. The key idea is to look at the linear
interpolant $\varsigma_t \, : \, t\mapsto (1-t)\mu + t\,\nu$, $t\in[0,1]$
and to find a corresponding family of diffeomorphisms $\Phi_t$ such
that $(\Phi_t)_*d\mu = d\varsigma_t$. (Here, as before, the $(\cdot)_*$ 
notation, applied to a measure, means
``push-forward'', i.e. $f_*d\mu
=d\nu$ is equivalent to the requirement that, for every Borel set
$\Omega $, $\nu(f(\Omega))=\mu(\Omega)$.) Then the projection is defined
as $f_m := \phi_m = \Phi_1$.

Dacaronga and Moser \cite{Dacarogna1990} used this strategy to
construct an area-preserving map that takes a given density $f$ to a
constant density. We will slightly generalize their formulation to
achieve an area-preserving mapping $\phi_m:\D\too\D$ taking the area
element $d\mu$ to $d\nu$, that is
\begin{equation}\label{e:push_forward_measure_cond}
    (\phi_{m})_*  d\mu = d\nu.
\end{equation}

Other researchers have used Moser's technique to construct an initial
guess in the further elaboration of an area-preserving map that would be
optimal in the sense of mass-transportation cost \cite{Dominitz10}.
Although Monge's mass-transportation provides a
elegant way to construct correspondence maps,
we believe that, because Euclidean (or hyperbolic) distances in the
uniformization plane have no intrinsic meaning for the geometry of
the problem, using them in the present context will not give a more
meaningful answer than the straightforward result of Moser's
procedure. More meaningful would be to use the surfaces' induced
geodesic distances in a mass-transportation approach, but this is a
much more challenging project, which we intend to tackle in future
work.

Since our measures are absolutely continuous w.r.t to the Lebesgue
measures $dz$, $dw$ respectively, we write $d\mu=\mu(z)dz$, $d\nu =
\nu(w) dw$, using the conformal factors $\mu(z),\nu(w)$ as
densities. Using the standard change of variables formula
we see that (\ref{e:push_forward_measure_cond}) can be rewritten, in
terms of the densities, as
\begin{equation}\label{e:volpres_formulated_on_densities}
    \nu(\phi_m(z))\det(\nabla \phi_m) = \mu(z).
\end{equation}

We will be interested in a solution to
(\ref{e:volpres_formulated_on_densities}) that is a diffeomorphism
$\phi_m$ of $ \D$ onto itself;  in particular points on the boundary
of the unit disk should be mapped to the boundary again. For the
remainder of this subsection we will drop the subscript on $\phi_m$,
writing it as $\phi$ for brevity.

Adapting Dacorogna and Moser's procedure \cite{Dacarogna1990} we
define the diffeomorphism $\phi$ by integrating, for $t \in [0,1]$,
a special time dependent vector field $v_t(z)$ (to be defined
below):
\begin{align}\label{e:deformation_phi_t_based_on_v_t}
 \frac{d}{dt}\Phi_t(z) &= v_t(\Phi_t(z)) ,& \text{ for all } t\geq0, \ z \in \D \\
\Phi_0(z) &= z ,& \text{ for all } z \in \mathcal{D}.
\end{align}
The desired map $\phi$ is then the end result of the integration,
$\phi(z)=\Phi_1(z)$. The vector field $v_t$ is defined in three
steps, as follows. We start by solving a Poisson equation with
Neumann boundary conditions,
\begin{align}
 \Delta a &= \mu - \nu ,& \text{in } \mathcal{D} \label{e:eq_for_a}\\
\frac{\partial a}{\partial n} &= 0 ,& \text{on } \partial
\mathcal{D} \,.\label{e:neumann_boundary_cond}
\end{align}
[Note that, unlike Dacorogna and Moser we do not require $\phi (z) =
z$ for $z\in \partial \D$;  we impose only that the boundary of $\D$
be mapped to the boundary -- hence the use of  Neumann instead of
Dirichlet boundary conditions.] Next, a time-{\em{independent}}
vector field $v$ is defined by setting $v(z) = \nabla a(z)$. In the
third step, we define the  time-{\em{dependent}} vector field $v_t$
as
\begin{equation*}
 v_t(z) = \frac{v(z)}{t \cdot \nu(z) + (1-t) \cdot \mu(z)}.
\end{equation*}
Establishing that $\phi(z)=\Phi_1(z)$ provides a solution to
(\ref{e:volpres_formulated_on_densities}) can be done by adapting
Dacorogna and Moser's original proof. For completeness let us
briefly describe the argument. First, we define an auxiliary function:
\begin{equation}\label{lambda}
 \lambda(t,z) = \Big(\det{\nabla \Phi_t(z)}\Big) \,\Big( t \cdot \nu(\Phi_t(z))
 + (1-t) \cdot \mu(\Phi_t(z))\Big) \,;
\end{equation}
as we shall see below, this function satisfies
\begin{equation}\label{lambdader}
 \frac{\partial}{\partial t} \lambda(t,z) = 0.
\end{equation}
For the time derivative of the first factor we refer to
\cite{Dacarogna1990}:
\begin{equation}
 \frac{\partial}{\partial t} (\det{\nabla \Phi_t(z)}) = \det{\nabla
 \Phi_t(z)} \cdot \divergence v_t(\Phi_t(z)).
\end{equation}
Differentiating \eqref{lambda} w.r.t. time gives thus
\begin{align}\label{e:dldt}
 \frac{\partial}{\partial t} \lambda(t,x) =& \det{\nabla \Phi_t} \cdot \divergence v_t(\Phi_t) \cdot \big(t \cdot \nu(\Phi_t) + (1-t) \cdot \mu(\Phi_t) \big)\\
& + \det{\nabla \Phi_t}\Big(\nu(\Phi_t) - \mu(\Phi_t) + \langle
t \nabla \nu(\Phi_t) + (1-t) \nabla
\mu(\Phi_t),\frac{d}{dt}\Phi_t \rangle \Big). \nonumber
\end{align}
By the definition of $v_t$ we obtain
\begin{equation*}
 \divergence v = (\divergence v_t)\, \big( t \cdot \nu + (1-t) \cdot \mu \big) + \langle t \nabla \nu + (1-t) \nabla \mu, v_t \rangle.
\end{equation*}
Together with \eqref{e:deformation_phi_t_based_on_v_t} this leads to
several cancellations in (\ref{e:dldt}), resulting in
\begin{equation}
 \frac{\partial}{\partial t} \lambda(t,z) = \det{\nabla \Phi_t} \big( \divergence v(\Phi_t) + (\nu(\Phi_t)-\mu(\Phi_t))\big).
\end{equation}
Since $v$ is defined as $\nabla a$, and $a$ satisfies
\eqref{e:eq_for_a}, this implies \eqref{lambdader}. Therefore,
\begin{equation*}
 \lambda(0,z) = \lambda(1,z).
\end{equation*}
Because $\Phi_0(z)=z$, we have $\lambda(0,z) = \mu(z)$, so that
we have shown that
\begin{equation*}
 \mu(z) = \det{\nabla \Phi_1(z)} \,\nu(\Phi_1(z)).
\end{equation*}

Finally, it is clear from the Neumann boundary conditions
(\ref{e:neumann_boundary_cond}) that the vector field $v(z)$ and
therefore $v_t(z)$ is tangent to the unit circle at the boundary of
the unit disk, that is $\ip{v_t(z),z}=0$ for all $z\in \partial\D$
and $t\geq 0$. 
This property ensures
that integral curves $\Phi_t(z)$ for $z\in \partial \D$ will stay on
the boundary of the disk for all times $t\geq 0$.

{\bf Implementation details:}
We used the $\mbox{{\sc Matlab}}^{\mbox{\tiny{\sc{tm}}}}$
\textsc{pde toolbox} for all steps. For the first step (solving the
Poisson equation) we used a triangular mesh with regular mesh size.
The two densities are taken to be piecewise constant on the
elements, with constants given by evaluating $\mu$ and $\nu$ at
the midpoints of the triangles, providing the right-hand side of the
PDE. Since the solution $a$ of the PDE is also piecewise constant on
the mesh elements, its gradient field $v = \grad a$ can be
determined on each node of the mesh. By a nearest neighbor
interpolation we approximate $v$ as piecewise constant on the
elements and use this to solve the ODE in the second step. This is
done with a 4-stage Runge-Kutta method. Another implementation
detail is that we add a small constant to the densities to avoid
numerical inabilities for densities that have a minimal value close
to zero.

\subsection{Thin-Plate Splines deformation.}
\label{subsec_4.3}
From a practical point of view we found it desirable to define our projection map
as a composition of {\em two} maps: $f_m=\phi_m\circ \zeta_m$, combining the Moser map
$\phi_m$ defined above with a 
a preliminary smooth planar deformation $\zeta_m$.
The map $\zeta_m$ is used to locally align the peaks and valleys, already brought
close together by the M\"{o}bius transformation $m$. Since
we assume the two surfaces have equal (unit) area,
i.e., $\int_\M d\vol_\M(x) = 1 = \int_\N d\vol_\N(y) $, improving the alignment
of peaks and valleys of the densities leads to less
area distortion. This quick-and-dirty approximation jumpstarts the
transition towards an exact area-preservation; although
true area-preservation is achieved only after the second step of
the deformation, an initial alignment
by means of $\zeta_m$ removes some of the ``workload'' on $\phi_m$. 

For the smooth deformation $\zeta_m$, we use Thin-Plate Splines (TPS).
In a first step, we label the points in $I_\M$ and $I_\N$ as follows. 
We first apply $m$ to the set $I_\M$, determine mutually closest
points (with respect to the hyperbolic distance function) for
the two sets $m\parr{I_\M}$ and $I_\N$, and label them correspondingly, so
that $(p_j,q_j)\in
m\parr{I_\M}\times I_\N, j=1,...,n, $ denote the mutually closest
pairs. In other words, we have
$$d_H(p_j,q_j) < \min \set{ \mathop{\min}_{q_j \ne q\in I_\N} d_H(p_j,q),
\mathop{\min}_{p_j \ne p\in I_\M} d_H(p,q_j)   },$$ where the
hyperbolic distance is
$d_H(p,q)=\tanh^{-1}\abs{\frac{p-q}{1-p\bar{q}}},$ $p,q \in \D$.

Next, we carry out a change of coordinates that maps the unit disk to the whole plane, 
by setting $\chi(z) = \mathrm{atan}(\abs{z})\,z/\abs{z}$ with the inverse
$\chi^{-1}(z) = \mathrm{tan}(\abs{z})\,z/\abs{z}$. Set $P_j=\chi(p_j)$,
$Q_j=\chi(q_j)$, $j=1, \ldots,n$. 
We construct a thin-plate spline function $\zeta_m$ interpolating
the $P_j$ and $Q_j$ in the complex plane, i.e., $\zeta_m(P_j)=Q_j$.
More explicitly,
$$\zeta_m(z) = \chi^{-1} \circ TPS_m \circ \chi,$$
where
$$TPS_m(z) = a_0+a_1 z + a_2 \bar{z} + \sum_{i=1}^n b_i \Upsilon(\abs{z-P_j}),$$
and $\Upsilon(r) = r^2 \log(r)$. The coefficients $a_j,b_i$,
$j=0,1,2, i=1,..,n$ are computed in the standard way by solving an $(n+3)\times (n+3)$
linear system \cite{wendland2005scattered} that imposes 
$TPS_m(P_j)=Q_j$, $j=1, \ldots,n$. ``Sandwiching'' $TPS_m$ by the coordinate transformation 
$\chi$ guarantees that $\zeta_m$ 
takes the disk $\D$ onto itself.

\subsection{Numerical experiments}
Figures \ref{fig:mesh_deformation_for_volpres} and \ref{fig:final_map}  demonstrate different aspects of
the behavior of the algorithm described in the earlier sections.

\begin{figure}[h]
\includegraphics[height=7 cm]{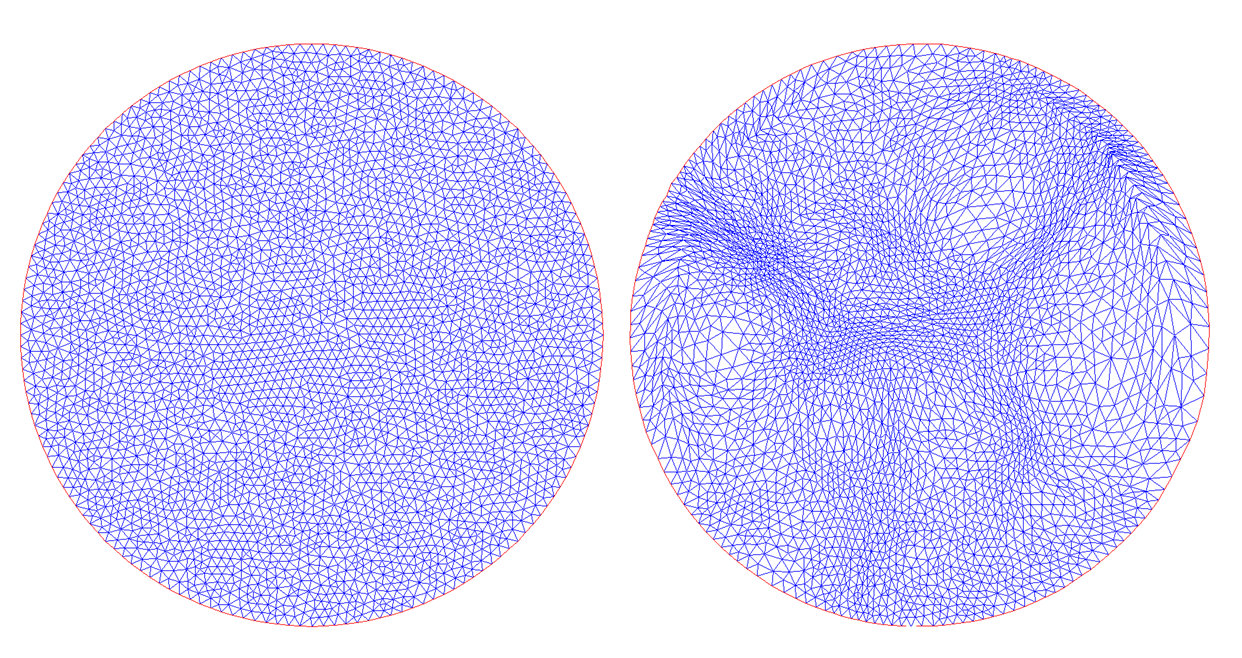} 
  \caption{{\bf Mesh deformation under Moser's procedure.}
Left: a "jiggled" initial mesh; Right: its deformation after solving the PDE on this mesh.}\label{fig:mesh_deformation_for_volpres}
\end{figure}

\begin{figure}[h]
   \includegraphics[width=\textwidth]{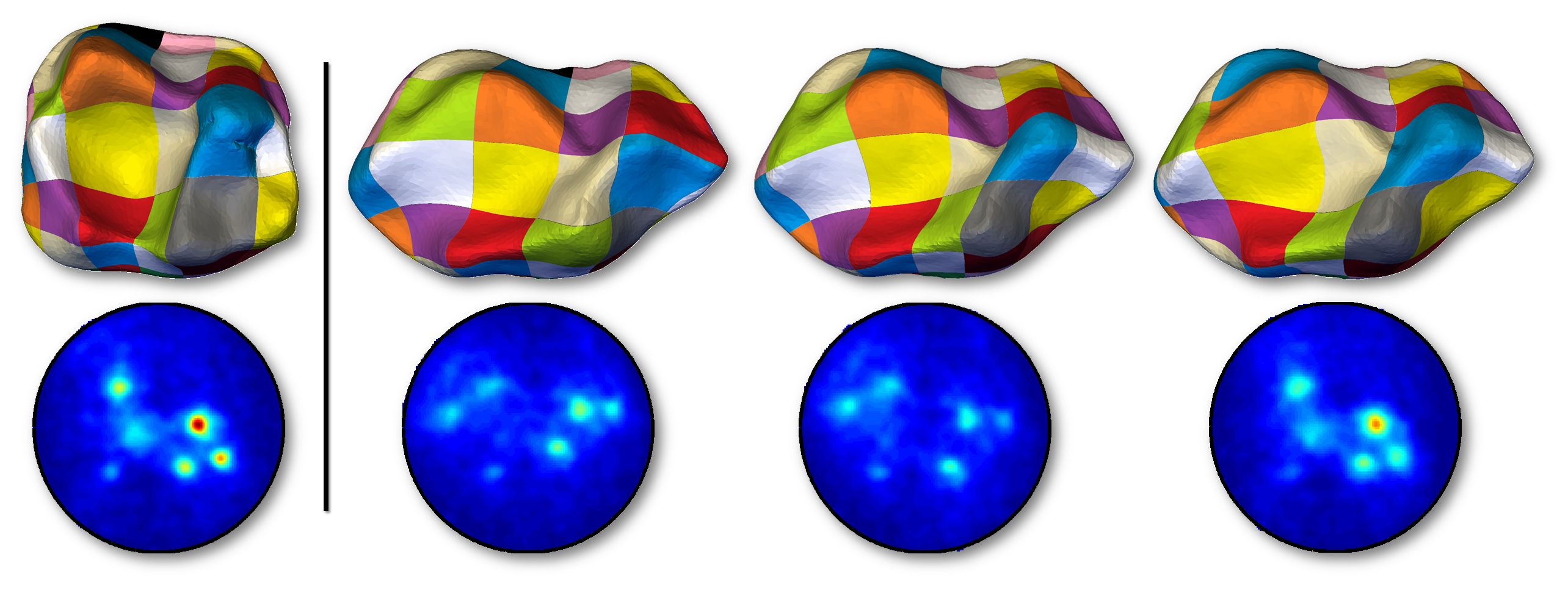} 
\caption{{\bf The different components of the map.}
The left column shows the surface $\N$ (top, with colored squares 
enabling the viewer to track where different 
portions of the surface $\M$ are mapped), and the corresponding conformal factor $\nu$
on the disk $\D$. The different steps of the
algorithm ``at work'' in constructing the map
$\C$ from $\M$ to $\N$ are shown to the right of the vertical black line. 
From left to right: the optimal M\"{o}bius transformation
$m$, which gives a uniformization of $\M$, with conformal factor $\mu$; 
the optimal alignment of the peaks in $\mu$, via TPS, with those of $\nu$;
the transformation into a truly area-preserving map via Moser's technique. 
}\label{fig:final_map}  
\end{figure}

In the companion paper \cite{pnas} an extensive analysis is
performed for three biological data-sets, comparing the results of several
algorithms to define the (dis)similarity between surfaces with those obtained by human experts. 
One of the methods illustrated in \cite{pnas} uses the algorithm described here, and 
we refer the interested reader to that paper for many more figures and results. 
(In the interest of full disclosure, we confess that in many of the examples 
in \cite{pnas} that used continuous Procrustes distances, we skipped the 
last step in $\pr$: the combination of an optimal M\"{o}bius
transformation and TPS already gave results that were very close to area-preserving, and
sufficed for the application at hand, so that we could skip the more time-consuming
Moser transformation.)

\section*{Acknowledgments}
The authors gratefully acknowledge partial support by NSF and AFOSR grants. 
\bibliographystyle{amsplain}
\bibliography{surface_procrustes}
\end{document}